\documentclass[leqno, 12pt]{amsart}
\usepackage{amsmath,amsfonts,amsthm,amssymb,indentfirst,mathrsfs}
\usepackage{xy} \xyoption{all}

\newtheorem{theorem}{Theorem}
\newtheorem{lemma}[theorem]{Lemma}
\newtheorem{definition}[theorem]{Definition}
\newtheorem{corollary}[theorem]{Corollary}
\newtheorem{proposition}[theorem]{Proposition}

\theoremstyle{definition}
\newtheorem{example}[theorem]{Example}

\newcommand{\End}{\mathrm{End}}

\newcommand{\tr}{\mathrm{tr}}

\newcommand{\path}{\mathrm{Path}}
\newcommand{\clpath}{\mathrm{ClPath}}
\newcommand{\R}{\mathbb{R}}

\newcommand{\C}{\mathbb{C}}
\newcommand{\N}{\mathbb{N}}
\newcommand{\M}{\mathbb{M}}

\newcommand{\so}{\mathbf{s}}
\newcommand{\ra}{\mathbf{r}}

\begin{document}

\title{Traces on Semigroup Rings and Leavitt Path Algebras}
\author{Zachary Mesyan and Lia Va\v{s}}

\address{Zachary Mesyan, Department of Mathematics, University of Colorado, Colorado Springs, CO 80918, USA }
\email{zmesyan@uccs.edu}

\address{Lia Va\v{s}, Department of Mathematics, Physics and Statistics, University of the Sciences, Philadelphia, PA 19104, USA}
\email{l.vas@usciences.edu}

\begin{abstract}
The trace on matrix rings, along with the augmentation map and Kaplansky trace on group rings, are some of the many examples of linear functions on algebras that vanish on all commutators. We generalize and unify these examples by studying traces on (contracted) semigroup rings over commutative rings. We show that every such ring admits a minimal trace (i.e., one that vanishes \emph{only} on sums of commutators), classify all minimal traces on these rings, and give applications to various classes of semigroup rings and quotients thereof. We then study traces on Leavitt path algebras (which are quotients of contracted semigroup rings), where we describe all linear traces in terms of central maps on graph inverse semigroups and, under mild assumptions, those Leavitt path algebras that admit faithful traces.

\medskip


\end{abstract}

\maketitle

\section{Introduction}

Perhaps the most familiar example of a linear map $t$ that is central (i.e., satisfies $t(ab)=t(ba)$ for all $a,b$) is the trace on a matrix ring. Other frequently used central linear functions include the augmentation map and the Kaplansky trace on group rings. Various generalizations of the above maps also have been explored in the literature on ring theory, module theory, and operator algebras.

One of the goals of this paper is to provide a general framework for studying a variety of central linear maps, or \emph{traces}. We begin by characterizing linear traces on (contracted) semigroup rings over commutative rings (Proposition~\ref{C-linear-trace}), which include both group rings and matrix rings, and hence unify the examples mentioned above. We then show that every (contracted) semigroup ring over a commutative ring admits a minimal linear trace (i.e., one whose kernel consists entirely of sums of commutators and is hence as small as possible), and characterize all minimal linear traces on such rings in terms of an equivalence relation on the underlying semigroups (Theorem~\ref{min-trace-exists}). This result is motivated by the well-known fact that the usual trace on a matrix ring over a commutative ring is minimal. We then apply our result to matrix rings (Corollary~\ref{usual-tr-min}), group rings (Corollary~\ref{gp-ring-tr}), and Cohn path algebras (Corollary~\ref{Cohn-min-tr}). Along the way we describe all central maps on various classes of semigroups, and in particular on graph inverse semigroups (Proposition~\ref{central-graph-maps}).

The remainder of the paper is devoted to linear traces on quotients of semigroup rings, especially on Leavitt path algebras. We describe when (minimal) linear traces on (contracted) semigroup rings over commutative rings pass to (minimal) linear traces on quotients thereof (Proposition~\ref{traces-on-quotients}). We then classify all linear traces on an arbitrary Leavitt path algebra in terms of central maps on the underlying graph inverse semigroup (Theorem~\ref{LPA-lin-tr}). 

Finally, we turn to faithful traces on Leavitt path algebras. (A trace that maps nonzero positive elements, with respect to the involution on the ring, to nonzero positive elements is said to be \emph{faithful}.) Our interest in faithful traces is motivated by the well-known fact that the usual trace on a matrix ring over a commutative ring with a positive definite involution is faithful. Moreover, Leavitt path algebras (first introduced in~\cite{AP,AMP}) are algebraic analogues of graph $C^*$-algebras, on which faithful traces have been studied extensively (e.g.,~\cite{PR, PRS}). Assuming that $K$ is a field with a positive definite involution and $E$ is a row-finite graph where every infinite path ends either in a sink or in a cycle, we show that the corresponding Leavitt path algebra $L_K(E)$ admits a faithful trace if and only if $E$ has no cycles with exits (Theorem~\ref{faithful-LPA}). We also give examples to illustrate the necessity of the hypotheses on $K$ and $E$ in this result. Other results about traces on Leavitt path algebras are obtained in~\cite{Vas}.

The necessary notions about semigroups, graphs, and Leavitt path algebras are reviewed along the way.

\subsection*{Acknowledgements} 

We are grateful to Gene Abrams and Gonzalo Aranda Pino for very helpful discussions about this subject, particularly regarding Corollary~\ref{LPA-matrix}. We would also like to thank the referee for a very thoughtful review.

\section{Traces on Semigroup Rings} \label{def-sect}

We begin by explaining our notation and describing the setting for our investigation of traces.

All rings will be assumed to be associative but not necessarily unital. Given a ring $R$ and elements $x,y \in R$, $[x,y]$ will denote the commutator $xy-yx$, and $[R,R]$ will denote the additive subgroup of $R$ generated by the commutators. 

\begin{definition} \label{trace-def}
Let $R$ and $T$ be rings. A $T$-\emph{valued trace on} $R$ is a map $t: R \to T$ satisfying $t(x+y)=t(x)+t(y)$ and $t(xy)=t(yx)$ for all $x,y \in R$.

If $R$ and $T$ are $C$-algebras, for some commutative ring $C$, then $t$ is $C$-\emph{linear} in case $t(cx)=ct(x)$ for all $x \in R$ and $c \in C$.

We say that $t$ is \emph{minimal} if for all $x \in R$, $t(x)=0$ implies that $x\in [R,R]$.
\end{definition}

Our usage of ``minimal" in the above definition is justified by the fact that $t(x)=0$ for any trace $t: R \to T$ and any $x \in [R,R]$.

Given a ring $R$ and a semigroup $G$ (with zero), we denote by $RG$ the corresponding semigroup ring, and by $\overline{RG}$ the corresponding \emph{contracted semigroup ring}, where the zero of $G$ is identified with the zero of $RG$. That is, $\overline{RG} = RG/I$, where $I = \{x\cdot 0_G \in RG \mid x \in R\}$ is the ideal of $RG$ generated by the zero $0_G$ of $G$. We note that if $G$ is a semigroup without zero, and $G^0$ is the semigroup obtained from $G$ by adjoining a zero element, then $RG \cong \overline{RG^0}$. 

For the remainder of this note, ``semigroup" will be understood to mean a semigroup with zero. We shall use the notation $\sum_g a_gg$ to denote an arbitrary element of a contracted semigroup ring $\overline{RG}$, where it is understood that $g$ ranges over all the nonzero elements of $G$, $a_g \in R$, and all but finitely many of the $a_g$ are zero. Finally, given a map $\delta: G \to R$ between semigroups, we shall say that $\delta$ \emph{preserves zero} if it takes zero to zero.

\begin{definition} \label{semigp-trace-def}
Let $C$ be a commutative ring, $R$ a $C$-algebra, $G$ a semigroup, and $\delta: G \to R$ a map that preserves zero. 

We denote by $t_\delta$ the map $\overline{CG} \to R$ defined by $t_\delta (\sum_g a_gg) = \sum_g a_g\delta(g)$, for all $a_g \in C$ and $g \in G$.

We say that $\delta$ is \emph{central} if $\delta(gh) = \delta(hg)$ for all $g,h \in G$, and that $\delta$ is \emph{normalized} if $R$ is unital and $\delta(g) \in \{0,1\}$ for all $g \in G$.
\end{definition}

With the help of the above definition we can describe all linear traces on (contracted) semigroup rings over commutative unital rings. This description, while very simple, will be useful throughout the paper.

\begin{proposition}\label{C-linear-trace}
Let $C$ be a commutative ring, $R$ a $C$-algebra, $G$ a semigroup, and $\delta : G \to R$ a central map that preserves zero. Then $t_\delta : \overline{CG} \to R$ is a $C$-linear trace. Moreover, if $C$ is unital, then every $C$-linear trace $t: \overline{CG} \to R$ is of this form.
\end{proposition}

\begin{proof}
Let $x, y \in \overline{CG}$ be arbitrary elements, and write $x=\sum_g a_gg$, $y=\sum_g b_gg$ for some $a_g,b_g \in C$. Then $$t_\delta(x+y) = t_\delta \Big(\sum_g (a_g + b_g)g\Big) = \sum_g (a_g + b_g)\delta(g) = \sum_g a_g\delta(g) + \sum_g b_g\delta(g) = t_\delta(x) + t_\delta(y),$$ and, using the commutativity of $C$ and the fact that $\delta$ is central, we have  $$t_\delta(xy) = t_\delta \Big(\sum_g \sum_h a_gb_hgh \Big) = \sum_g \sum_h a_gb_h\delta(gh) = \sum_h \sum_g b_ha_g\delta(hg) = t_\delta (yx).$$ Thus $t_\delta$ is an $R$-valued trace, and it is clearly $C$-linear.

For the final claim, suppose that $C$ is unital and $t: \overline{CG} \to R$ is a $C$-linear trace. Let $x=\sum_g a_gg \in \overline{CG}$ be any element. Then $$t(x)=t\Big(\sum_g a_gg\Big)=\sum_g a_gt(g),$$ since $t$ is $C$-linear. (Here we identify each $g\in G \setminus \{0\}$ with $1 \cdot g \in \overline{CG}$.) Also, for any $g,h \in G$ we have $t(gh) = t(hg)$, since $t$ is a trace, and hence the restriction of $t$ to $G$ gives a central map, which necessarily preserves zero. Thus $t$ has the desired form.
\end{proof}

Here are some familiar examples of (normalized) traces of the above form.

\begin{example} \label{kap-eg}
Let $C$ be a commutative unital ring, $G$ a group, and $G^0$ the semigroup obtained by adjoining a zero to $G$. Define $\delta : G^0 \to C$ by $\delta (e) =1$, where $e \in G^0$ is the identity element, and $\delta(g)=0$ for all $g \in G^0\setminus \{e\}$. Then it is easy to see that $\delta$ is central. The map $t_\delta : CG \cong \overline{CG^0} \to C$ is known as the \emph{Kaplansky trace}.
\end{example}

\begin{example} \label{augmentation-eg}
Let $C$ be a commutative unital ring, $G$ a group, and $G^0$ the semigroup obtained by adjoining a zero to $G$. Define $\delta : G^0 \to C$ by $\delta (g) =1$ for all $g \in G^0 \setminus \{0\}$, and $\delta(0)=0$. Then $\delta$ is clearly central, and $t_\delta : CG \cong \overline{CG^0} \to C$ is called \emph{the augmentation map}. 
\end{example}

\begin{example} \label{matrix-example}
Let $C$ be a commutative unital ring and $n \geq 1$ an integer. Then the ring $\M_n(C)$ of $n\times n$ matrices over $C$ is isomorphic to the contracted semigroup ring $\overline{CG}$, where $G=\{e_{ij} \mid 1\leq i,j \leq n\} \cup \{0\}$, and $e_{ij}$ are the matrix units. Define $\delta : G \to C$ by $\delta(e_{ij}) = \delta_{ij}$, where $\delta_{ij}$ is the Kronecker delta, and $\delta(0)=0$. Then for all $1 \leq i,j,k,l \leq n$ we have $\delta(e_{ij}e_{kl}) =1$ if and only if $l=i$, $k=j$ if and only if $\delta(e_{kl}e_{ij}) = 1$. It follows that $\delta$ is a central map. It is now easy to see that $t_\delta : \M_n(C) \to C$ is precisely the usual trace on $\M_n(C)$.
\end{example} 

\begin{example} \label{inf-matrix-example}
Let $C$ be a commutative unital ring, $\kappa$ an infinite cardinal, and $\M_\kappa (C)$ the ring of infinite matrices over $C$, having rows and columns indexed by $\kappa$, with only finitely many nonzero entries. Then $\M_\kappa(C)$ is isomorphic to $\overline{CG}$, where $G=\{e_{ij} \mid i,j \in \kappa\} \cup \{0\}$, and $e_{ij}$ are the matrix units. Defining $\delta : G \to C$ as in Example~\ref{matrix-example} again gives a central map, and hence $t_\delta : \M_\kappa (C) \to C$ is a $C$-linear trace.
\end{example} 

\section{Minimal Traces} \label{min-section}

It is well known that the linear traces in Examples~\ref{matrix-example} and~\ref{inf-matrix-example} are minimal (see e.g.,~\cite[Corollary 17]{ZM1}). In this section we shall generalize this fact by constructing minimal linear traces for all semigroup rings over commutative unital rings and then classifying the minimal linear traces on such rings. The following relation will help us with this task.

\begin{definition} \label{sim-def}
Let $G$ be a semigroup and $g,h \in G$. We shall write $g \sim h$ if either $g=h$ or there are elements $a_1, b_1, a_2, b_2, \dots, a_n, b_n \in G$ such that $$g=a_1b_1, \ \ b_1a_1=a_2b_2, \ \ b_2a_2=a_3b_3, \dots, \ \ b_{n-1}a_{n-1}=a_nb_n, \ \ b_na_n=h.$$
\end{definition}

The relation $\sim$ is clearly an equivalence. We shall denote the $\sim$-equivalence class of $g \in G$ by $[g]$. We note that if $G\setminus \{0\}$ is a group, and $g,h \in G\setminus \{0\}$, then $g \sim h$ if and only if $g=fhf^{-1}$ for some $f \in G\setminus \{0\}$. Thus $\sim$ generalizes the notion of conjugacy from groups to all semigroups.

Before proceeding to our construction, let us give an example of a semigroup $G$ and elements $g,h \in G$ such that $g \sim h$, but where there do not exist $a, b \in G$ satisfying $g=ab$, $ba=h$. This will show that in general $\sim$ does not reduce to the simpler relation where $g$ and $h$ are related if $g=h$ or $g=ab$, $ba=h$ for some $a, b \in G$.

\begin{example}
Let $S=\{1, 2, 3, 4\}$, and let $G = \End (S)$ be the semigroup of all set maps from $S$ to itself. We define $c, d, e, f \in G$ as follows.
$$c: 
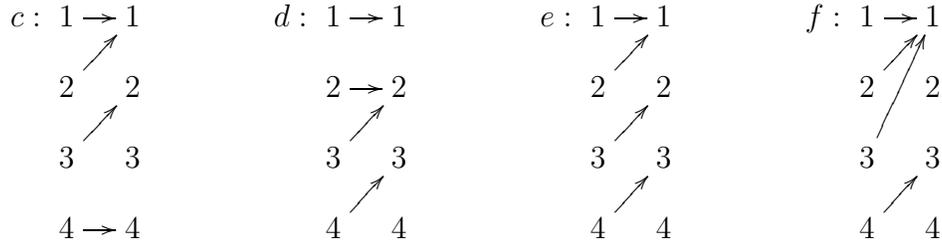
\xymatrix@=1pc{
1 \ar[r] & 1\\
2 \ar[ur] & 2\\
3 \ar[ur] & 3\\
4 \ar[r] & 4
} \ \ \ \ \ \ \ \ \ \ \ \ d: \xymatrix@=1pc{
1 \ar[r] & 1\\
2 \ar[r] & 2\\
3 \ar[ur] & 3\\
4 \ar[ur] & 4
}  \ \ \ \ \ \ \ \ \ \ \ \ e: \xymatrix@=1pc{
1 \ar[r] & 1\\
2 \ar[ur] & 2\\
3 \ar[ur] & 3\\
4 \ar[ur] & 4
} \ \ \ \ \ \ \ \ \ \ \ \ f: \xymatrix@=1pc{
1 \ar[r] & 1\\
2 \ar[ur] & 2\\
3 \ar[uur] & 3\\
4 \ar[ur] & 4
}$$ Then, letting elements of $G$ act on $S$ from the left, we have the following composites.
$$dc: 
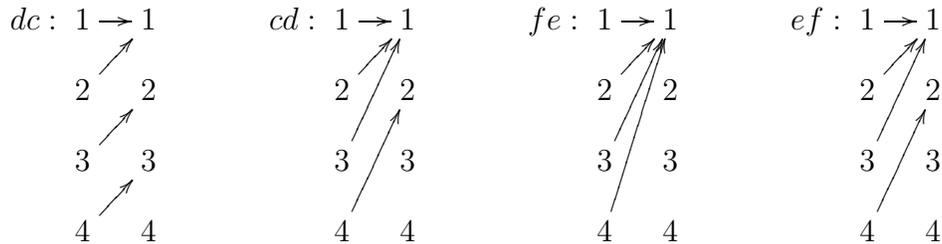
\xymatrix@=1pc{
1 \ar[r] & 1\\
2 \ar[ur] & 2\\
3 \ar[ur] & 3\\
4 \ar[ur] & 4
} \ \ \ \ \ \ \ \ \ \ cd: \xymatrix@=1pc{
1 \ar[r] & 1\\
2 \ar[ur] & 2\\
3 \ar[uur] & 3\\
4 \ar[uur] & 4
}  \ \ \ \ \ \ \ \ \ \ fe: \xymatrix@=1pc{
1 \ar[r] & 1\\
2 \ar[ur] & 2\\
3 \ar[uur] & 3\\
4 \ar[uuur] & 4
} \ \ \ \ \ \ \ \ \ \ ef: \xymatrix@=1pc{
1 \ar[r] & 1\\
2 \ar[ur] & 2\\
3 \ar[uur] & 3\\
4 \ar[uur] & 4
}$$ Thus, $cd=ef$, and hence letting $g=dc$ and $h=fe$, we see that $g\sim h$. 

Clearly, for all $a,b \in G$ we have $|ab(S)| \leq |a(S)|, |b(S)|$ (where for any set $X$, $|X|$ denotes the cardinality of $X$). Hence, if $ab=g$, then $|a(S)|, |b(S)| \geq 3$. But then $|ba(S)| \geq 2$, and therefore $ba \neq h$. Thus, there do not exist $a, b \in G$ such that $g=ab$ and $ba=h$.
\end{example}

Extending a familiar notion from linear algebra, given a commutative ring $C$ and a $C$-algebra $R$, we say that a subset $S$ of $R$ is \emph{$C$-linearly independent} if for all $x_1, \dots, x_n \in S$ and all $c_1,\dots, c_n \in C$, $c_1x_1 + \dots + c_nx_n = 0$ implies that $c_1=\dots = c_n = 0$.

In the following lemma we give our minimal trace construction.

\begin{lemma} \label{min-trace-lemma}
Let $C$ be a commutative unital ring, $R$ a $C$-algebra, $G$ a semigroup, and $S$ the set of all nonzero $\sim$-equivalence classes of $G$. Suppose that $\, \{r_{[g]} \mid [g] \in S\} \subseteq R$ is a $C$-linearly independent set. Then defining $\delta : G \to R$ by $$\delta (g) = 
\left\{ \begin{array}{ll}
r_{[g]} & \text{if } \, g \not\sim 0\\
0 & \text{if } \, g \sim 0
\end{array}\right.,$$ gives a central map, and $t_\delta : \overline{CG} \to R$ is a minimal $C$-linear trace.
\end{lemma}

\begin{proof}
For any $g, h \in G$, we have $gh \sim hg$, and hence $\delta (gh) = r_{[gh]} = r_{[hg]} = \delta (hg)$. Thus $\delta$ is a central map that preserves zero, and by Proposition~\ref{C-linear-trace}, $t_\delta : \overline{CG} \to R$ is a $C$-linear trace. To show that it is minimal, first note that if $h\sim f$ for some distinct $h,f \in G$, then we can find $a_1, b_1, a_2, b_2, \dots, a_n, b_n \in G$ such that $$h=a_1b_1, \ \ b_1a_1=a_2b_2, \ \ b_2a_2=a_3b_3, \dots, \ \ b_{n-1}a_{n-1}=a_nb_n, \ \ b_na_n=f,$$ and therefore (using the fact that $C$ is unital) $$h-f = a_1b_1 - b_1a_1 + a_2b_2-b_2a_2 + \dots + a_nb_n - b_na_n \in [\overline{CG}, \overline{CG}].$$ In particular, this implies that for all $c_h \in C$, $$\sum_{h \in [0]} c_hh = \sum_{h \in [0]} c_h(h-0) \in [\overline{CG}, \overline{CG}].$$ 

To conclude the proof, let us suppose that $t_\delta (x)=0$ for some $x \in \overline{CG}$, and show that $x \in [\overline{CG}, \overline{CG}]$. We can write $x = \sum_{[g] \in S \cup \{[0]\}} \sum_{h \in [g]} c_hh$ for some $c_h \in C$. Since $$0=t_\delta (x) =  \sum_{[g] \in S \cup \{[0]\}} \sum_{h \in [g]} c_h \delta (h) = \sum_{[g] \in S} \sum_{h \in [g]} c_h r_{[g]} = \sum_{[g] \in S} \Big( \sum_{h \in [g]} c_h \Big) r_{[g]},$$ and since the $r_{[g]}$ are $C$-linearly independent, we see that $\sum_{h \in [g]} c_h = 0$ for all $[g] \in S$. It is enough to show that $\sum_{h \in [g]} c_hh \in [\overline{CG}, \overline{CG}]$ for each $[g] \in S \cup \{[0]\}$ such that $c_h \neq 0$ for some $h \in [g]$. Further, by the previous paragraph, we may assume that $[g] \neq [0]$. Thus, let $[g] \in S$ and assume that $c_f \neq 0$ for some $f \in [g]$. Then $c_f = -\sum_{h \in [g] \setminus \{f\}} c_h$, and hence $$\sum_{h \in [g]} c_hh = \sum_{h \in [g] \setminus \{f\}} c_hh - \Big( \sum_{h \in [g] \setminus \{f\}} c_h \Big) f = \sum_{h \in [g] \setminus \{f\}} c_h(h-f).$$ Since, as noted above, $h-f \in [\overline{CG}, \overline{CG}]$ for each $h \in [g] \setminus \{f\}$, we have $\sum_{h \in [g]} c_hh \in [\overline{CG}, \overline{CG}]$, as desired.
\end{proof}

Given a ring $C$ and a set $S$, we denote by $C^{(S)}$ the direct sum of $|S|$ copies of $C$, indexed by the elements of $S$. If $S= \emptyset$, then we understand $C^{(S)}$ to be the zero ring.

We are now ready for our main result about minimal traces.

\begin{theorem} \label{min-trace-exists}
Let $C$ be a commutative unital ring, $G$ a semigroup, and $S$ the set of all nonzero $\sim$-equivalence classes of $G$. Then the following hold.
\begin{enumerate}
\item[$(1)$] There is a minimal $C$-linear trace $t : \overline{CG} \to C^{(S)}$.
\item[$(2)$] If $R$ is a $C$-algebra, and $t: \overline{CG} \to R$ is a $C$-linear trace, then $t$ is minimal if and only if for all $g_1, \dots, g_n \in G$ such that $\, [g_1], \dots, [g_n] \in S$ are distinct, $t(g_1), \dots, t(g_n)$ are $C$-linearly independent in $R$.
\item[$(3)$] There is a normalized minimal $C$-linear trace $t : \overline{CG} \to R$ for some unital $C$-algebra $R$ if and only if $\, |S| \leq 1$.
\end{enumerate}
\end{theorem}

\begin{proof}
To show (1), define $\delta : G \to C^{(S)}$ by $$\delta (g) = 
\left\{ \begin{array}{ll}
e_{[g]} & \text{if } \, g \not\sim 0\\
0 & \text{if } \, g \sim 0
\end{array}\right.,$$ where $e_{[g]} \in C^{(S)}$ denotes the element with $1$ as the entry in the coordinate indexed by $\, [g]$ and zeros elsewhere. The set $\{e_{[g]} \mid [g] \in S\}$ is clearly $C$-linearly independent, and hence, by Lemma~\ref{min-trace-lemma}, $t_\delta : \overline{CG} \to C^{(S)}$ is a minimal $C$-linear trace. 

For (2), let $t: \overline{CG} \to R$ be a $C$-linear trace. Suppose that $g_1, \dots, g_n \in G$ are such that $[g_1], \dots, [g_n] \in S$ are distinct, but $c_1t(g_1)+ \dots +c_nt(g_n) = 0$ for some $c_1, \dots, c_n \in C$, not all of which are zero. Then $t(c_1g_1 + \dots + c_ng_n) = 0$, but $$t_\delta (c_1g_1 + \dots + c_ng_n) = c_1e_{[g_1]} + \dots + c_ne_{[g_n]} \neq 0$$ (where $t_\delta$ is the map from the previous paragraph), and hence $c_1g_1 + \dots + c_ng_n \notin [\overline{CG}, \overline{CG}]$, showing that $t$ is not minimal. 

Conversely, suppose that for all $g_1, \dots, g_n \in G$ such that $[g_1], \dots, [g_n] \in S$ are distinct, $t(g_1), \dots, t(g_n)$ are $C$-linearly independent in $R$. By Proposition~\ref{C-linear-trace}, we have $t=t_\delta$, where $\delta : G \to R$ is the restriction of $t$ to $G$. Set $r_{[g]} : = \delta ([g])$ for all $g \in G$ such that $g \not\sim 0$. (Since $\delta$ is central, it must agree on all elements of a $\sim$-equivalence class.) Then $\{r_{[g]} \mid [g] \in S\}$ is a $C$-linearly independent set, by hypothesis, and therefore it follows from Lemma~\ref{min-trace-lemma} that $t$ is minimal.

For (3), we note that if $t : \overline{CG} \to R$ is a minimal $C$-linear trace that is normalized on $G$, then $|t(G)| \leq 2$. Hence, by (2), there can be at  most one nonzero $\sim$-equivalence class in $G$ (i.e., $|S| \leq 1$). Conversely, if $|S| \leq 1$, then defining $\delta : G \to C$ by $$\delta (g) = 
\left\{ \begin{array}{ll}
1 & \text{if } \, g \not\sim 0\\
0 & \text{if } \, g \sim 0
\end{array}\right.$$
gives a minimal $C$-linear trace $t_\delta : \overline{CG} \to C$, by Lemma~\ref{min-trace-lemma}, which is clearly normalized.
\end{proof}

Let us next give several consequences of the above theorem, along with related observations.

\begin{corollary} \label{commutator-rings}
Let $C$ be a commutative unital ring, $R$ a nonzero $C$-algebra, and $G$ a semigroup. Then the following are equivalent.
\begin{enumerate}
\item[$(1)$] The only central map $\delta: G \to R$ that preserves zero is the zero map.
\item[$(2)$] For every $f \in G$ we have $f\sim 0$.
\item[$(3)$] $\overline{CG} = [\overline{CG}, \overline{CG}]$.
\item[$(4)$] The only trace $t: \overline{CG} \to R$ is the zero map.
\item[$(5)$] The only $C$-linear trace $t: \overline{CG} \to R$ is the zero map.
\end{enumerate}
\end{corollary}

\begin{proof}
$(1) \Rightarrow (2)$ Let $X \subseteq G$ consist of all the elements $f \in G$ such that $f \not\sim 0$, and let $r \in R\setminus \{0\}$. If $X \neq \emptyset$, then $$\delta (f) = 
\left\{ \begin{array}{ll}
r & \text{if } f \in X\\
0 & \text{if } f \notin X
\end{array}\right.$$
defines a nonzero map $\delta: G \to R$ that preserves zero. Moreover, $\delta$ is central, since for all $g,h \in G$ we have $gh \sim hg$, and hence $$\delta (gh) = 0 \iff gh \sim 0 \iff hg \sim 0 \iff \delta (hg) = 0.$$ Thus, if $(1)$ holds, then $X=\emptyset$, and therefore $(2)$ must also hold. 

$(2) \Rightarrow (3)$ If for every $f \in G$ we have $f\sim 0$, then by Theorem~\ref{min-trace-exists}(1), there is a minimal trace from $\overline{CG}$ to the zero ring, which can only happen if $\overline{CG} = [\overline{CG}, \overline{CG}]$.

$(3) \Rightarrow (4)$ Clearly, if every element of $\overline{CG}$ is a sum of commutators, then any trace on $\overline{CG}$ must send every element to zero.

$(4) \Rightarrow (5)$ This is a tautology.

$(5) \Rightarrow (1)$ This follows from Proposition~\ref{C-linear-trace}.
\end{proof}

In the following example we construct a nonzero semigroup satisfying statement (2) of Corollary~\ref{commutator-rings}. We obtain it by modifying the semigroups $G$ in Examples~\ref{matrix-example} and~\ref{inf-matrix-example}.

\begin{example}
Let $G=\{e_{ij} \mid i,j \in \R, i<j\} \cup \{0\}$, where $\R$ is the set of the real numbers (though, any dense well-ordered set would do in place of $\R$). Define multiplication on $G$ by 
$$e_{ij}\cdot e_{kl} = 
\left\{ \begin{array}{ll}
e_{il} & \text{if } j = k\\
0 & \text{if } j \neq k
\end{array}\right.$$ 
and $0\cdot e_{ij} = 0 = e_{ij}\cdot 0$ for all $i,j,k,l \in \R$. It is easy to see that this operation is associative, and hence that $(G, \cdot)$ is a semigroup. Now, let $e_{ij} \in G \setminus \{0\}$ be any element, and let $k \in \R$ be any number such that $i < k < j$. Then $e_{ij} = e_{ik}e_{kj}$, but $e_{kj}e_{ik} = 0$. Hence $e_{ij} \sim 0$, implying that all $f \in G$ satisfy $f \sim 0$.
\end{example}

By Corollary~\ref{commutator-rings}, if $C$ is any commutative unital ring and $G$ is the semigroup constructed above, then $\overline{CG} = [\overline{CG}, \overline{CG}]$. For other examples of rings $R$ satisfying $R = [R,R]$ see~\cite{ZM1, ZM2}.

Using Theorem~\ref{min-trace-exists} we can strengthen the well-known fact that the usual trace on a matrix ring is minimal. See Example~\ref{inf-matrix-example} for the definition of $\M_\kappa (C)$, when $\kappa$ is infinite.

\begin{corollary} \label{usual-tr-min}
Let $C$ be a commutative unital ring, $\kappa$ a nonzero cardinal, and $\, \tr : \M_\kappa (C) \to C$ the usual trace. Then a $C$-linear map $t : \M_\kappa (C) \to C$ is a trace if and only if there exists $c \in C$ such that $t(M) = c \cdot \tr (M)$ for all $M \in \M_\kappa (C)$. In this case, $t$ is minimal if and only if $c$ is not a zero divisor.
\end{corollary}

\begin{proof}
Let $G=\{e_{ij} \mid i,j \in \kappa\} \cup \{0\}$, where $e_{ij}$ are the matrix units. Then $\M_\kappa(C)$ is isomorphic to the contracted semigroup ring $\overline{CG}$, as noted in Examples~\ref{matrix-example} and~\ref{inf-matrix-example}. Letting $i,j \in \kappa$, we see that if $i \neq j$, then $e_{ij} = e_{ij}e_{jj}$ and $e_{jj}e_{ij}=0$, while $e_{ii}=e_{ij}e_{ji}$ and $e_{ji}e_{ij}=e_{jj}$. It follows that $G$ has only one nonzero $\sim$-equivalence class, namely $\{e_{ii} \mid i \in \kappa\}$. Thus, for any central map $\delta : G \to C$ that preserves zero there must be some $c \in C$ such that
$$\delta (e_{ij}) = 
\left\{ \begin{array}{ll}
c & \text{if } i = j\\
0 & \text{if } i \neq j
\end{array}\right..$$
Noting that for such $\delta$ we have $t_\delta = c \cdot \tr$, the first claim now follows from Proposition~\ref{C-linear-trace}. 

The second claim follows from Theorem~\ref{min-trace-exists}(2), since $\{t (e_{ii})\} = \{c\}$ is $C$-linearly independent if and only if $ac \neq 0$ for all $a \in C \setminus \{0\}$, i.e., $c$ is not a zero divisor.
\end{proof}

Next, let us turn to traces on group rings. By Theorem~\ref{min-trace-exists}, for any group $G$ and commutative unital ring $C$, $CG$ admits a minimal $C$-linear trace. However, the usual traces on $CG$ (see Examples~\ref{kap-eg} and~\ref{augmentation-eg}) do not have this property, as the next result shows.

\begin{corollary} \label{gp-ring-tr}
Let $C$ be a commutative unital ring and $G$ a nontrivial group. Then there are no minimal $C$-linear traces $t: CG \to C$.
\end{corollary}

\begin{proof}
As noted in Section~\ref{def-sect}, we may identify $CG$ with $\overline{CG^0}$, where $G^0$ is the semigroup obtained from $G$ by adjoining a zero element. By Theorem~\ref{min-trace-exists}(2), if there is a minimal $C$-linear trace $t:  \overline{CG^0} \to C$, then there must be only one $\sim$-equivalence class of elements of $G$. But, the $\sim$-equivalence class of the identity $e$ consists of just one element, and hence this can happen only if $G=\{e\}$.
\end{proof}

We conclude this section by noting that in certain situations our theory of minimal traces on semigroup rings extends to quotients of such rings.

\begin{proposition} \label{traces-on-quotients}
Let $C$ be a commutative ring, $R$ a $C$-algebra, $G$ a semigroup, $\delta : G \to R$ a central map that preserves zero, and $I \subseteq \overline{CG}$ an ideal. If $t_\delta (I) = 0$, then $\bar{t}_\delta (r+I) = t_\delta (r)$ defines a $C$-linear trace $\bar{t}_\delta : \overline{CG}/I \to R$. Moreover, $t_\delta$ is minimal if and only if $\bar{t}_\delta$ is minimal and $I \subseteq [\overline{CG}, \overline{CG}]$.
\end{proposition}

\begin{proof}
Since $t_\delta (I) = 0$, $\bar{t}_\delta$ is well-defined, and it is routine to verify that $\bar{t}_\delta$ is a $C$-linear trace. For instance, for all $x,y \in \overline{CG}$ we have $$\bar{t}_\delta ((x+I)(y+I))=\bar{t}_\delta(xy+I)=t_\delta (xy) = t_\delta (yx) = \bar{t}_\delta (yx+I) = \bar{t}_\delta ((y+I)(x+I)).$$

Now, suppose that $t_\delta$ is minimal. Since $t_\delta (I) = 0$, we must have $I \subseteq [\overline{CG}, \overline{CG}]$. Also, if $\bar{t}_\delta (x+I)=0$ for some $x \in \overline{CG}$, then $t_\delta (x) =0$, and hence $x \in [\overline{CG},\overline{CG}]$. Thus, $x+I \in [\overline{CG}/I, \overline{CG}/I]$, implying that $\bar{t}_\delta$ is minimal. Conversely, suppose that $\bar{t}_\delta$ is minimal and $I \subseteq [\overline{CG}, \overline{CG}]$. If $t_\delta (x) =0$ for some $x \in \overline{CG}$, then $\bar{t}_\delta (x+I)=0$, and hence $x+I \in [\overline{CG}/I, \overline{CG}/I]$. Since $I \subseteq [\overline{CG}, \overline{CG}]$, this implies that $x \in [\overline{CG},\overline{CG}]$. Therefore, $t_\delta$ is minimal.
\end{proof}

\section{Graph Inverse Semigroups}\label{semigp-section}

In this section we describe the central maps and $\sim$-equivalence classes in certain semigroups arising from graphs. In addition to supplying further examples to which Proposition~\ref{C-linear-trace} and Theorem~\ref{min-trace-exists} can be applied, this will help us study traces on Leavitt path algebras in subsequent sections. We begin by recalling some notions pertaining to semigroups and graphs.

A semigroup $S$ is an \emph{inverse semigroup} if for each $x \in S$ there is a unique element $x^{-1} \in S$, called the \emph{semigroup inverse} of $x$, satisfying $x = xx^{-1}x$ and $x^{-1} = x^{-1}xx^{-1}$.

A \emph{directed graph} $E=(E^0,E^1,\so, \ra)$ consists of two  sets $E^0,E^1$ (the elements of which are called \emph{vertices} and \emph{edges}, respectively), together with functions $\so,\ra:E^1 \to E^0$, called \emph{source} and \emph{range}, respectively. We shall refer to directed graphs as simply ``graphs" from now on. A \emph{path} $p$ in $E$ is a finite sequence $e_1\cdots e_n$ of (not necessarily distinct) edges $e_1,\dots, e_n \in E^1$ such that $\ra(e_i)=\so(e_{i+1})$ for $i\in \{1,\dots,n-1\}$. Here we define $\so(p):=\so(e_1)$ to be the \emph{source} of $p$, $\ra(p):=\ra(e_n)$ to be the \emph{range} of $p$, and $n$ to be the \emph{length} of $p$. We view the elements of $E^0$ as paths of length $0$, and denote by $\path(E)$ the set of all paths in $E$. A path $p = e_1\cdots e_n$ is said to be \emph{closed} if $\so(p)=\ra(p)$. Such a path is said to be a \emph{cycle} if in addition $\so(e_i)\neq \so(e_j)$ for every $i\neq j$. A cycle consisting of just one edge is called a \emph{loop}. We denote the set of all closed paths in $\path(E)$ by $\clpath(E)$. A graph which contains no cycles (other than vertices) is called \emph{acyclic}. An edge $e \in E^1$ is an \emph{exit} for a path $p = e_1\cdots e_n$ if $\so(e) = \so(e_i)$ and $e \neq e_i$ for some $i\in \{1,\dots,n\}$. A graph where no cycle has an exit is called a \emph{no-exit graph}.

A vertex $v \in E^0$ for which the set $\so^{-1}(v) = \{e\in E^1 \mid \so(e)=v\}$ is finite is said to have \emph{finite out-degree}. A graph $E$ is said to have \emph{finite out-degree}, or to be \emph{row-finite}, if every vertex of $E$ has finite out-degree. A vertex $v \in E^0$ which is the source of no edges of $E$ is called a \emph{sink}, while a vertex having finite out-degree which is not a sink is called \emph{regular}. A graph $E$ where both $E^0$ and $E^1$ are finite, respectively countable, sets is called a \emph{finite}, respectively \emph{countable}, graph.

\begin{definition} \label{graph-semigp-def}
Given a graph $E=(E^0,E^1,\so, \ra)$, the \emph{graph inverse semigroup $G_E$ of $E$} is the semigroup with zero generated by the sets $E^0$ and $E^1$, together with a set of variables $\{e^* \mid e\in E^1\}$, satisfying the following relations for all $v,w\in E^0$ and $e,f\in E^1$:

\emph{(V)}  $vw = \delta_{v,w}v$,

\emph{(E1)} $\so(e)e=e\ra(e)=e$,

\emph{(E2)} $\ra(e)e^*=e^*\so(e)=e^*$,

\emph{(CK1)} $e^*f=\delta _{e,f}\ra(e)$.
\end{definition}

For all $v \in E^0$ we define $v^*:=v$, and for all paths $p = e_1 \cdots e_n$ ($e_1, \dots, e_n \in E^1$) we set $p^*:=e_n^* \cdots e_1^*$, $\ra(p^*):=\so(p)$, and $\so(p^*):=\ra(p)$. With this notation, every nonzero element of $G_E$ can be written uniquely as $pq^*$ for some $p, q \in \path(E)$, by the CK1 relation, and multiplication of two such elements $pq^*$ and $rs^*$ is given by $$pq^*rs^*=\left\{
\begin{array}{ll}
pts^* & \mbox{if }r=qt\mbox{ for some } t \in \path(E)\\
pt^*s^* & \mbox{if }q=rt\mbox{ for some } t \in \path(E)\\ 
0 & \mbox{otherwise}
\end{array}\right..$$
It is easy to verify that $G_E$ is indeed an inverse semigroup, where the semigroup inverse of a nonzero element $pq^*$ is $qp^*$. We also note that for any graph $E$, extending $\so$ and $\ra$ to $E^1 \cup \{e^* \mid e\in E^1\}$ as above turns $(E^0,E^1 \cup \{e^* \mid e\in E^1\},\so, \ra)$ into a (directed) graph, which is called the \emph{extended graph of $E$}. 

Graph inverse semigroups were first introduced in~\cite{AH}, in order to show that every partially ordered set can be realized as the partially ordered set of nonzero $\mathscr{J}$-classes of an inverse semigroup. (Two elements in a semigroup are $\mathscr{J}$-\emph{equivalent} if they generate the same ideal.) This class of semigroups has since been studied in its own right (e.g., \cite{JL, Krieger, Paterson}).

The following notation will help us classify all central maps on graph inverse semigroups that preserve zero (and hence also all linear traces on contracted semigroup rings arising from them, by Proposition~\ref{C-linear-trace}).

\begin{definition} \label{approx-def}
Let $E$ be a graph and $p, q \in \clpath(E)$. We write $p \approx q$  if there exist paths $x, y \in \path(E)$ such that $p=xy$ and $yx=q$.
\end{definition}

It is shown in \cite[Lemma 12]{ZM2} that $\approx$ is an equivalence relation. Observe that if $q=e_1\cdots e_n \in \clpath (E)\setminus E^0$ for some $e_1, \dots, e_n \in E^1$, then the $\approx$-equivalence class of $q$ consists of $e_1 \cdots e_n, e_2 \cdots e_ne_1, \dots, e_ne_1 \cdots e_{n-1}$. The $\approx$-equivalence class of a vertex consists of just one element. We shall discuss the connection between $\approx$ and the relation $\sim$ on $G_E$ from Definition~\ref{sim-def} in Proposition~\ref{eq-description} below.

\begin{proposition} \label{central-graph-maps}
Let $R$ be a ring, $E$ a graph, and $\delta : G_E \to R$ a map that preserves zero. Then $\delta$ is central if and only if the following conditions hold.
\begin{enumerate}
\item[$(1)$] If $\delta (x) \neq 0$ for some $x \in G_E$, then either $x=pqp^*$ or $x=pq^*p^*$ for some $p \in \path (E)$ and $q \in \clpath (E)$.
\item[$(2)$] For all $p \in \path (E)$ and $q \in \clpath (E)$, we have $\delta(pqp^*)=\delta(q)$ and $\delta(pq^*p^*)=\delta(q^*)$.
\item[$(3)$] For all $p,q \in \clpath (E)$ such that $p \approx q$, we have $\delta (p) = \delta (q)$ and $\delta (p^*) = \delta (q^*)$.
\end{enumerate}
\end{proposition}

\begin{proof}
Suppose that $\delta$ is central, and let $p \in \path (E)$ be such that $\so(p) \neq \ra(p)$. Then $\so(p)p = p$, $p\so(p)=0$, $p^*\so(p) = p^*$, and $\so(p)p^*=0$. We therefore must have $\delta (p) = 0 = \delta (p^*)$.

Now, let $p,s \in \path (E)$ be such that $\delta (ps^*) \neq 0$. Then $s^*p \neq 0$, and hence either $s = pq$ or $p=sq$ for some $q \in \path (E)$. If $s = pq$, then $$0 \neq \delta (ps^*) = \delta (pq^*p^*) = \delta (q^*p^*p) = \delta (q^*),$$ and hence $\so(q)=\ra(q)$, by the previous paragraph. If $p = sq$, then $$0 \neq \delta (ps^*) = \delta (sqs^*) = \delta (s^*sq) = \delta (q),$$ and hence $\so(q)=\ra(q)$, as before. In either case, $q$ is a closed path, and therefore $x=ps^*$ has the form described in (1). Clearly, $\delta$ being central implies (2) and (3).

Conversely, suppose that (1), (2), and (3) hold. Let $p,q,r,s \in \path(E)$, and suppose that $\delta (pq^*rs^*) \neq 0$. We wish to show that $\delta (pq^*rs^*) = \delta (rs^*pq^*)$. By (1), $pq^*rs^* = ghg^*$ or $pq^*rs^* = gh^*g^*$ for some $g \in \path(E)$ and $h \in \clpath (E)$. Let us assume that $pq^*rs^* = ghg^*$, since the other situation can be handled analogously. Thus, there must be some $v \in \path(E)$ such that either $q=rv$ or $r=qv$. In the first case, $$ghg^*= pq^*rs^* = pv^*r^*rs^* = pv^*s^*,$$ which implies that $g = sv$ and $p = gh = svh$. Hence using (2), we have $$\delta (pq^*rs^*) = \delta (svhv^*r^*rs^*) = \delta(svhv^*s^*) = \delta (h)$$ $$= \delta(rvhv^*r^*)  = \delta(rs^*svhv^*r^*) = \delta (rs^*pq^*).$$ We may therefore suppose that $r=qv$. Then $$ghg^*= pq^*rs^* = pq^*qvs^* = pvs^*,$$ which implies that $s = g$ and $pv = gh = sh$. Thus there must be some $u \in \path(E)$ such that either $p=su$ or $s=pu$. In the first case $h = uv$, and hence, using (2) and (3), we have $$\delta (pq^*rs^*) = \delta (suq^*qvs^*) = \delta (suvs^*) = \delta (uv)$$ $$= \delta (vu) = \delta (qvuq^*) = \delta (qvs^*suq^*) = \delta (rs^*pq^*).$$ Let us therefore assume that $r=qv$ and $s=pu$. Then $$ghg^*= pq^*rs^* = pq^*qvu^*p^* = pvu^*p^*,$$ which implies that $g=pu$ and therefore $v=uh$. Hence, again using (2), we have $$\delta (pq^*rs^*) = \delta (pvu^*p^*) = \delta (puhu^*p^*) = \delta (h)$$ $$= \delta (quhu^*q^*) = \delta (qvu^*q^*) = \delta (qvu^*p^*pq^*) = \delta (rs^*pq^*).$$ 

Thus if $\delta (pq^*rs^*) \neq 0$, then $\delta (pq^*rs^*) = \delta (rs^*pq^*)$. By symmetry, it also follows that if $\delta (rs^*pq^*) \neq 0$, then $\delta (rs^*pq^*) = \delta (pq^*rs^*)$. Hence, $\delta (pq^*rs^*) = \delta (rs^*pq^*)$ for all values of $\delta (pq^*rs^*)$ and $\delta (rs^*pq^*)$, and therefore $\delta$ is central.
\end{proof}

\begin{proposition} \label{eq-description}
Let $E$ be a graph, and for each $q \in \clpath (E)$ set 
$$EQ(q) := \{ptp^* \mid p \in \path (E), t \in \clpath (E), \ra(p)=\so(t), t \approx q\} \text{ and}$$
$$EQ(q^*) := \{pt^*p^* \mid p \in \path (E), t \in \clpath (E), \ra(p)=\ra(t), t \approx q\}.$$
Then every nonzero $\sim$-equivalence class of $G_E$ $($in the sense of Definition~\ref{sim-def}$)$ is of the form $EQ(q)$ or $EQ(q^*)$ for some $q \in \clpath (E)$.

In particular, for all $q_1, q_2 \in \clpath (E)$ we have $EQ(q_1) \cap EQ(q_2) \neq \emptyset$ if and only if $q_1 \approx q_2$ if and only if $EQ(q_1^*) \cap EQ(q_2^*) \neq \emptyset$, and $EQ(q_1) \cap EQ(q_2^*) \neq \emptyset$ if and only if $q_1=q_2 \in E^0$.
\end{proposition}

\begin{proof}
Let $q \in \clpath (E)$ and $p_1t_1p_1^*, p_2t_2p_2^* \in EQ(q)$ be any elements. Since $t_1 \approx q \approx t_2$, there exist $x,y \in \path (E)$ such that $t_1 = xy$ and $t_2 = yx$. Hence $$p_1t_1p_1^* \sim p_1^*p_1t_1 = t_1 = xy \sim yx = t_2 = p_2^*p_2t_2 \sim p_2t_2p_2^*,$$ and therefore all elements of $EQ(q)$ are $\sim$-equivalent. A similar argument shows the analogous statement for $EQ(q^*)$.

Next, let $x \in G_E$ be any element such that $x \not\sim 0$. We wish to show that $x \in EQ(q) \cup EQ(q^*)$ for some $q \in \clpath (E)$. Let $C$ be any commutative unital ring. Then, by Theorem~\ref{min-trace-exists}(1), there is a minimal $C$-linear trace $t : \overline{CG_E} \to R$, for some $C$-algebra $R$. By Proposition~\ref{C-linear-trace}, $t=t_\delta$ for some central map $\delta : G_E \to R$ that preserves zero. Since $t$ is minimal and $x \not\sim 0$, by Theorem~\ref{min-trace-exists}(2), $\delta (x) \neq 0$. Hence, by Proposition~\ref{central-graph-maps}, $x = pqp^*$ or $x = pq^*p^*$ for some $p \in \path (E)$ and $q \in \clpath (E)$. Therefore either $x \in EQ(q)$ or $x \in EQ(q^*)$.

To conclude the proof of the first claim it now suffices to show that for any $q \in \clpath (E)$ and $x \in G_E$, $x \sim q$ implies that $x \in EQ(q)$, while $x \sim q^*$ (and $q \notin E^0$) implies that $x \in EQ(q^*)$. Let $C$ be any commutative unital ring, let $S$ be the set of $\approx$-equivalence classes of $\clpath (E)$, and let $S^*$ be another copy of $S$. Define $\delta : G_E \to C^{(S\cup S^*)}$ by
$$\delta (y) = \left\{ \begin{array}{cl}
\epsilon_{(q)} & \text{if } y \in EQ(q) \text{ for some } q\in \clpath(E)\\
\epsilon_{(q)^*} & \text{if } y \in EQ(q^*) \text{ for some } q\in \clpath(E)\setminus E^0\\
0 & \textrm{otherwise}
\end{array}\right.,$$ where $\epsilon_{(q)} \in C^{(S \cup S^*)}$ denotes the element with $1 \in C$ in the coordinate indexed by the $\approx$-equivalence class of $q$ in $S$ and zeros elsewhere, while $\epsilon_{(q)^*} \in C^{(S \cup S^*)}$ denotes the element with $1 \in C$ in the coordinate indexed by the $\approx$-equivalence class of $q$ in $S^*$ and zeros elsewhere. We shall show below that $\delta$ is well defined, but assuming for the moment that this is established, by Proposition~\ref{central-graph-maps}, $\delta$ is a central map. It follows that if $x \sim q$, then $\delta (x) = \delta (q) = \epsilon_{(q)}$, and hence $x \in EQ(q)$. Similarly, if $x \sim q^*$ and $q \notin E^0$, then $\delta (x) = \delta (q^*) = \epsilon_{(q)^*}$, and hence $x \in EQ(q^*)$.

To show that the map $\delta$ is well defined, suppose that $y \in G_E$ belongs to $EQ(q_1)$ and to $EQ(q_2)$ for some $q_1, q_2 \in \clpath (E)$. Then $y = ptp^*$ for some $p \in \path (E)$ and $t \in \clpath (E)$, where $t \approx q_1$ and $t \approx q_2$. It follows that $q_1 \approx q_2$, and therefore $EQ(q_1)=EQ(q_2)$, showing that the assignment $\delta (y) = \epsilon_{(q_1)}$ is unambiguous. Similarly, if $y \in G_E$ belongs to $EQ(q_1^*)$ and to $EQ(q_2^*)$ for some $q_1, q_2 \in \clpath (E)$, then $q_1 \approx q_2$, and therefore $\delta (y) = \epsilon_{(q_1)^*}$ is well defined. Finally, if $y \in G_E$ belongs to $EQ(q_1)$ and to $EQ(q_2^*)$ for some $q_1, q_2 \in \clpath (E)$, then it must be the case that $q_1=q_2$ is a vertex. We have excluded $E^0$ in the second case of our definition of $\delta$ for this reason, and hence again, $\delta (y) = \epsilon_{(q_1)}$ is well defined.

In the course of showing that $\delta$ is well defined we have also proved one direction of each equivalence in the final claim. The other direction of each equivalence is trivial.
\end{proof}

Using Theorem~\ref{min-trace-exists} and Proposition~\ref{eq-description} we obtain a complete description of the minimal linear traces on the contracted semigroup ring $\overline{CG_E}$, for any commutative unital ring $C$ and graph $E$. In the case where $C$ is a field, $\overline{CG_E}$ is known as a \emph{Cohn path algebra}. 

\begin{corollary} \label{Cohn-min-tr}
Let $C$ be a commutative unital ring, $R$ a $C$-algebra, $E$ a graph, and $t: \overline{CG_E} \to R$ a $C$-linear trace. Then $t$ is minimal if and only if for all $q_1, \dots, q_n \in \clpath (E)$ and $p_1, \dots, p_m \in \clpath (E)\setminus E^0$, such that $q_i \not\approx q_j$ and $p_i \not\approx p_j$ for $i \neq j$, the elements $t(q_1), \dots, t(q_n),$  $t(p_1^*), \dots, t(p_m^*)$ are $C$-linearly independent in $R$.
\end{corollary}

\begin{proof}
We shall show that the following two statements are equivalent, from which the desired conclusion will follow, by Theorem~\ref{min-trace-exists}(2).
\begin{enumerate}
\item[$(1)$] For all $g_1, \dots, g_l \in G_E$, such that $[g_1], \dots, [g_l]$ are distinct nonzero $\sim$-equivalence classes of $G_E$, the elements $t(g_1), \dots, t(g_l)$ are $C$-linearly independent in $R$.
\item[$(2)$] For all $q_1, \dots, q_n \in \clpath (E)$ and $p_1, \dots, p_m \in \clpath (E)\setminus E^0$, such that $q_i \not\approx q_j$ and $p_i \not\approx p_j$ for $i \neq j$, the elements $t(q_1), \dots, t(q_n), t(p_1^*), \dots, t(p_m^*)$ are $C$-linearly independent in $R$.
\end{enumerate}

Suppose that (1) holds, and let $q_1, \dots, q_n, p_1, \dots, p_m$ be as in (2). By Proposition~\ref{eq-description}, the elements $q_1, \dots, q_n, p_1^*, \dots, p_m^*$ fall into distinct nonzero $\sim$-equivalence classes of $G_E$. Hence $t(q_1), \dots, t(q_n), t(p_1^*), \dots, t(p_m^*)$ must be $C$-linearly independent, by (1), and therefore (2) holds.

Conversely, suppose that (2) holds, and let $g_1, \dots, g_l \in G_E$ be such that $[g_1], \dots, [g_l]$ are distinct and nonzero. By Proposition~\ref{eq-description}, for each $g_i$ we can find some $q \in \clpath (E)$ such that $g_i \sim q$, or some $p \in \clpath (E) \setminus E^0$ such that $g_i \sim p^*$. Let us index the elements $q$ and $p$ selected as $q_1, \dots, q_n, p_1, \dots, p_m$. Since $[g_1], \dots, [g_l]$ are distinct, we must have $q_i \not\approx q_j$ and $p_i \not\approx p_j$ for $i \neq j$. Then, by (2), the elements $t(q_1), \dots, t(q_n), t(p_1^*), \dots, t(p_m^*)$ are $C$-linearly independent. Since $t$ is central, it must take the same value on all elements of a $\sim$-equivalence class, and hence $t(g_1), \dots, t(g_l)$ must be $C$-linearly independent as well, proving (1).
\end{proof}

\section{Involutions and Faithful Traces}

For the rest of this paper we shall be interested in traces on rings with involutions. Let us recall the relevant notions.

\begin{definition} \label{invol-def}
Let $R$ be a ring. A map $\, *:R\to R$ $($written using superscript notation$)$ is an \emph{involution} if $\, (x^*)^* = x$, $\, (x+y)^* = x^*+y^*$, and $\, (xy)^*=y^*x^*$ for all $x,y \in R$. If $R$ has an involution, then it is said to be a $\, *$-\emph{ring}.

If $R$ is a $\, *$-ring and $x \in R$ is a sum of elements of the form $yy^*$ $(y \in R)$, then $x$ is referred to as a \emph{positive} element, and this is denoted by $x\geq 0$, or by $x>0$, if $x \neq 0$. 

An involution $\, *$ on $R$ is \emph{positive definite} if for all $x_1, \dots, x_n \in R$, $\, \sum_{i=1}^n x_ix_i^* = 0$ implies that each $x_i=0$.

If $f: R \to T$ is a homomorphism of $\, *$-rings, then $f$ is a $\, *$-\emph{homomorphism} or an \emph{involution homomorphism} provided that $f(x^*)=f(x)^*$ for all $x \in R$. A $\, *$-\emph{isomorphism}, or an \emph{involution isomorphism}, is defined analogously.
\end{definition}

It is well known (see e.g.,~\cite[Section 50]{Berberian}) and easy to see that extending $\geq$ to all elements of $R$ via
\[x\geq y\mbox{ if and only if }x-y\geq 0\]
gives a reflexive and transitive relation. 

\begin{definition}
Let $R$ and $T$ be $\, *$-rings, and let $t : R \to T$ be a trace. We say that $t$ is \emph{positive} in case $t(x)\geq 0$ for all $x \in R$ with $x\geq 0$, and that $t$ is \emph{faithful} in case $t(x)> 0$ for all $x \in R$ with $x> 0$.
\end{definition}

Our primary interest henceforth will be in Leavitt path algebras, which constitute a particular class of quotients of semigroup rings. However, we conclude this section with several preliminary results of a more general nature, which will be required later. These are all easy to show and are mostly part of the folklore. We therefore only sketch the proofs.

Given $*$-rings $T_i$ $(i\in I)$, we understand the direct sum $\bigoplus_{i\in I} T_i$ to have the corresponding component-wise involution, where $(a_i)^*_{i \in I} = (a_i^*)_{i \in I}$ for every $(a_i)_{i \in I} \in \bigoplus_{i\in I} T_i$.

\begin{lemma} \label{prod-lemma}
Let $C$ be a commutative ring, and let $P$, $R$, $S$  be $C$-algebras with involutions.
\begin{enumerate}
\item[$(1)$] If $t : P \to R$ and $u : R \to S$ are $C$-linear positive $($respectively, faithful$)$ traces, then so is $u \circ t : P \to S$.
\item[$(2)$] If $\, \{T_i\}_{i \in I}$ is a collection of commutative $C$-subalgebras of $R$, and the involution $\, *$ on $R$ is positive definite, then $t((x_i)_{i \in I}) = \sum_{i \in I} x_i$ defines a faithful $C$-linear trace $t : \bigoplus_{i\in I} T_i \to R$.
\end{enumerate}
\end{lemma}

\begin{proof}
The verification of (1) is routine. 

Given $t$ as in (2), it is clear that $t$ is an additive $C$-linear map, and since each $T_i$ is commutative, $t$ is a trace. It is also positive, since for any $(x_i)_{i \in I} \in \bigoplus_{i\in I} T_i$, we have $$t((x_i)_{i \in I}(x_i)^*_{i \in I}) = t((x_ix_i^*)_{i \in I}) = \sum_{i \in I} x_ix_i^*.$$ That $t$ is faithful follows from the fact that the involution on $R$ is positive definite.
\end{proof}

\begin{lemma} \label{inv-matrix}
Let $C$ be a commutative ring with involution $\, *$, $\kappa$ an arbitrary nonzero cardinal, and $\, \tr: \M_\kappa (C) \to C$ the usual trace. Then $\, (a_{ij})^* := (a_{ji}^*)$ $((a_{ij}) \in \M_\kappa (C))$ defines an involution on $\, \M_\kappa (C)$, with respect to which $\, \tr$ is positive. Also, if $\, *$ is positive definite on $C$, then $\, \tr$ is faithful.
\end{lemma}

\begin{proof}
The fact that $*$ extends to an involution on $\M_\kappa (C)$ follows immediately from properties of $*$ on $C$ and the transpose operation on $\M_\kappa (C)$. Also, for any $A = (a_{ij}) \in \M_\kappa (C)$ we have $\tr(AA^*) = \sum_i\sum_j a_{ij}a_{ij}^*$, from which we conclude that $\tr$ is positive. If $*$ is positive definite on $C$, then this formula also implies that $\tr$ is faithful.
\end{proof}

\begin{lemma} \label{inv-laurent}
Let $C$ be a commutative ring with involution $\, *$. 
\begin{enumerate}
\item[$(1)$] Defining $\, (\sum_i a_i x^i)^* : = \sum_i a_i^* x^{-i}$ $(a_i \in C)$ gives an involution on $C[x,x^{-1}]$. 
\item[$(2)$] If $\, *$ is positive definite on $C$, then its extension to $C[x,x^{-1}]$ in $\, (1)$ is positive definite as well.
\item[$(3)$] If $\, *$ is positive definite on $C$, then defining $t : C[x,x^{-1}] \to C$ by $t(\sum_i a_i x^i) = a_0$ $(a_i \in C)$ gives a  faithful $C$-linear trace.
\end{enumerate}
\end{lemma}

\begin{proof}
Checking that $*$ extends to an involution on $C[x,x^{-1}]$ as in (1) is routine. 

To show (2), suppose that $*$ is positive definite on $C$, let $r_1, \dots, r_n \in C[x,x^{-1}]$, and suppose that $\sum_{i=1}^n r_ir_i^*=0$. Write each $r_i$ as $r_i = \sum_j a_{ij}x^j$ for some $a_{ij} \in C$. Then $$0 = \sum_{i=1}^n r_ir_i^* = \sum_{i=1}^n \bigg(\sum_j a_{ij}x^j\bigg)\bigg(\sum_j a_{ij}^*x^{-j}\bigg)  = \sum_{i=1}^n  \bigg(\sum_j a_{ij}a_{ij}^* + f_i(x)\bigg),$$ where each $f_i(x) \in C[x,x^{-1}]$ is a polynomial with $0$ as the degree-zero term. Thus, $0 = \sum_{i=1}^n \sum_j a_{ij}a_{ij}^*$, which implies that each $a_{ij}=0$, since $*$ is positive definite on $C$. Therefore, each $r_i=0$, showing that $*$ is positive definite on $C[x,x^{-1}]$.

For (3), suppose that $*$ is positive definite on $C$, and let $t$ be as in the statement. Then clearly, $t$ is a $C$-linear map, and since $C[x,x^{-1}]$ is commutative, it is a trace. The map $t$ is also positive, since for any $r=\sum_i a_i x^i \in C[x,x^{-1}]$ we have $$t(rr^*) = t\Big(\Big(\sum_ia_ix^{i} \Big)\Big(\sum_ia_i^*x^{-i} \Big)\Big) = t\Big(\sum_i\sum_j a_ia_j^* x^{i-j} \Big) = \sum_{i-j=0} a_ia_j^* = \sum_i a_ia_i^*.$$ The faithfulness of $t$ follows from $*$ being positive definite on $C$.
\end{proof}

\section{Traces on Leavitt Path Algebras}

In this section we describe all linear traces on Leavitt path algebras, and, under mild assumptions, the Leavitt path algebras that admit faithful traces. We begin by recalling the relevant definitions.

\begin{definition}\label{LPAdefinition}  
Let $K$ be a field, $E$ a graph, and $G_E$ the corresponding graph inverse semigroup. Then the contracted semigroup ring $\overline{KG_E}$ is known as the \emph{Cohn path $K$-algebra of $E$}, and is usually denoted by $C_K(E)$.
 
Let $N$ be the ideal of $C_K(E)$ generated by all elements of the form $v-\sum_{e\in \so^{-1}(v)} ee^*,$ where $v \in E^0$ is a regular vertex. Then $C_K(E)/N$ is called the \emph{Leavitt path $K$-algebra  of $E$}, and is denoted by $L_K(E)$.
\end{definition}

We note that Leavitt path algebras are typically defined  without reference to Cohn path algebras or graph inverse semigroups, as the $K$-algebras generated by the sets $\{v\mid v\in E^0\}$ and $\{e,e^*\mid e\in E^1\}$ (arising from a directed graph $E=(E^0,E^1,\so, \ra)$), which satisfy the conditions (V), (E1), (E2), and (CK1) from Definition~\ref{graph-semigp-def}, along with $v=\sum_{e\in \so^{-1}(v)} ee^*$ for all regular $v \in E^0$ (e.g., \cite{AP, AAPM, AM, AMP, ARV}).

From our description of $G_E$ in Section~\ref{semigp-section}, it follows that every element of $C_K(E)$ and $L_K(E)$ (by slightly abusing notation) can be expressed in the form $\sum_{i=1}^n a_ip_iq_i^*$ for some $a_i \in K$ and $p_i,q_i \in \path (E)$. It is known (see~\cite[Theorem 1.5.17]{AAM} or~\cite[Lemma 4.8]{Vas}) that every Cohn path $K$-algebra is isomorphic to some Leavitt path $K$-algebra. We shall therefore restrict our attention to Leavitt path algebras from now on.

If $K$ has an involution $*$, then $L_K(E)$ acquires one as well via $(\sum_{i=1}^n a_ip_iq_i^*)^* = \sum_{i=1}^n a_i^*q_ip_i^*$. It is shown in~\cite[Proposition 2.4]{ARV} that if the involution on $K$ is positive definite, then the same is true of the induced involution on $L_K(E)$.

Next, let us describe the linear traces on Leavitt path algebras. To do so we shall require~\cite[Lemma 9]{AM}, which says that letting $N \subseteq C_K(E)$ be as in Definition~\ref{LPAdefinition}, for every regular $v\in E^0$ and every $p \in \path (E)\setminus E^0$ we have $$\Big(v-\sum_{e\in \so^{-1}(v)} ee^*\Big)p= 0 = p^*\Big(v-\sum_{e\in \so^{-1}(v)} ee^*\Big).$$

\begin{theorem} \label{LPA-lin-tr}
Let $K$ be a field, $R$ a $K$-algebra, and $E$ a graph. Then there is a one-to-one correspondence between the $K$-linear traces $t: L_K(E) \to R$ and the maps $\delta : G_E \to R$ that preserve zero and satisfy the following conditions.
\begin{enumerate}
\item[$(1)$] If $\delta (x) \neq 0$ for some $x \in G_E$, then either $x=pqp^*$ or $x=pq^*p^*$ for some $p \in \path (E)$ and $q \in \clpath (E)$.
\item[$(2)$] For all $p \in \path (E)$ and $q \in \clpath (E)$, we have $\delta(pqp^*)=\delta(q)$ and $\delta(pq^*p^*)=\delta(q^*)$.
\item[$(3)$] For all $p,q \in \clpath (E)$ such that $p \approx q$, we have $\delta (p) = \delta (q)$ and $\delta (p^*) = \delta (q^*)$. $($See Definition~\ref{approx-def} for the relation $\, \approx$.$)$
\item[$(4)$] For all regular $v \in E^0$ we have $\delta (v) = \sum _{e\in \so^{-1}(v)} \delta (\ra(e))$.
\end{enumerate}
The one-to-one correspondence maps $K$-linear traces on $L_K(E)$ to their restrictions to $G_E$.
\end{theorem}

\begin{proof}
Let $t: L_K(E) \to R$ be a $K$-linear trace, and let $\delta=t|_{G_E}$ be the restriction of $t$ to $G_E$. Then $\delta$ clearly must preserve zero. Since $t$ is central, $\delta$ must be central as well, and hence must satisfy conditions (1)--(3), by Proposition~\ref{central-graph-maps}. Also, if $v \in E^0$ is regular, then, by Definition~\ref{LPAdefinition},  $v=\sum_{e\in \so^{-1}(v)} ee^*$. Hence $$\delta(v) = t\Big(\sum_{e\in \so^{-1}(v)} ee^*\Big) = \sum_{e\in \so^{-1}(v)} t(ee^*) = \sum_{e\in \so^{-1}(v)} \delta(e^*e) = \sum_{e\in \so^{-1}(v)} \delta(\ra(e)),$$ and therefore, $\delta$ satisfies (4) as well. Thus, $t \mapsto t|_{G_E}$ gives a map from the set of all $K$-linear traces $t: L_K(E) \to R$ to the set of all maps $\delta : G_E \to R$ that preserve zero and satisfy conditions (1)--(4).

Next, suppose that $\delta : G_E \to R$ is a map that preserves zero and satisfies conditions (1)--(4). Then $\delta$ is central, by Proposition~\ref{central-graph-maps}. Let $t_\delta : C_K(E) = \overline{KG_E} \to R$ be as in Definition~\ref{semigp-trace-def}. Then $t_\delta$ is a $K$-linear trace, by Proposition~\ref{C-linear-trace}. We wish to show that $t_\delta (N) = 0$, where $N \subseteq C_K(E)$ is as in Definition~\ref{LPAdefinition}. To do so, it suffices to prove that for every generator $x = v-\sum_{e\in \so^{-1}(v)} ee^*$ of $N$ and any two elements $y,z \in C_K(E)$, we have $t_\delta (yxz) = 0$. But since $t_\delta$ is a trace, $t_\delta (yxz) = t_\delta (zyx)$, and hence we only need to show that $t_\delta (yx) = 0$ for all $y \in C_K(E)$. As $t_\delta$ is $K$-linear, we may further assume that $y=pq^*$ for some $p,q \in \path (E)$. Again, since $t_\delta$ is a trace, we have $t_\delta(yx)=t_\delta(pq^*x)=t_\delta(q^*xp)$.  But by~\cite[Lemma 9]{AM}, the expression $q^*xp$ is zero unless $q^* = v = p$, in which case $q^*xp = x$.  Thus it suffices to show that $t_\delta (x) = 0$. Now $$t_\delta(x) = t_\delta\Big(v-\sum_{e\in \so^{-1}(v)} ee^*\Big) = \delta (v) - \sum_{e\in \so^{-1}(v)} \delta(ee^*) = \delta (v) - \sum_{e\in \so^{-1}(v)} \delta(\ra(e))=0,$$ since $\delta$ satisfies condition (4), and hence $t_\delta (N) = 0$.

Thus, defining $\bar{t}_\delta : L_K(E) = C_K(E)/N \to R$ via $\bar{t}_\delta (x+N) = t_\delta (x)$ gives a $K$-linear trace, by Proposition~\ref{traces-on-quotients}. Thus, $\delta \mapsto \bar{t}_\delta$ gives a function from the set of all maps $\delta : G_E \to R$ that preserve zero and satisfy conditions (1)--(4) to the set of all $K$-linear traces $t: L_K(E) \to R$. Note that the restriction of $\bar{t}_\delta$ to $G_E$ is $\delta$, from which it is easy to see that the two maps $t \mapsto t|_{G_E}$ and $\delta \mapsto \bar{t}_\delta$ defined above are inverses of each other (in either order). Thus we have the desired one-to-one correspondence.
\end{proof}

Our final goal is to describe Leavitt path algebras that admit faithful traces. We begin with a preliminary result.

\begin{proposition} \label{prop-of-graph-trace-lemma}
Let $K$ be a field with involution $*$, $R$ a $*$-ring, $E$ a graph, and $t: L_K(E) \to R$ a trace. If $t$ is positive, then the following hold.
\begin{enumerate}
\item[$(1)$] For all $v \in E^0$ we have $t (v) \geq 0$.
\item[$(2)$] For all $v, w \in E^0$, if there is a path $p \in \path (E)$ such that $\so(p) =v$ and $\ra(p) =w$, then $t(v) \geq t(w)$.
\item[$(3)$] For all $v \in E^0$ and distinct $e_1, \dots, e_n \in E^1$ with $v$ as the source, $t(v) \geq \sum_{i=1}^n t(\ra(e_i))$.
\end{enumerate}
Moreover, if $t$ is faithful, then the following hold.
\begin{enumerate}
\item[$(4)$]  For all $v \in E^0$ we have $t(v) > 0$.
\item[$(5)$]  For any $x \in L_K(E)$ and any idempotent $u\in L_K(E)$, satisfying $xu=x=ux$ and $x^*u=x^*=ux^*$, we have $xx^* = u$ if and only if $x^*x = u$.
\item[$(6)$]  $E$ is a no-exit graph.
\end{enumerate}
\end{proposition}

\begin{proof}
Suppose that $t$ is positive. To show (1), we note that $0 \leq t(vv^*) = t(v)$ for all $v \in E^0$. Next, let $v$, $w$, and $p$ be as in (2), and set $x = v -pp^*$. Then $xx^* = v -pp^*$, and therefore $$0 \leq t(xx^*) = t(v)-t(pp^*) = t(v) -  t(p^*p) = t(v)-t(w).$$

For (3), let $e_1, \dots, e_n \in E^1$ be distinct edges with source $v \in E^0$, and set $x=v - \sum_{i=1}^n e_ie_i^*$. Then $xx^*=v - \sum_{i=1}^n e_ie_i^*$, and therefore $$0\leq t(xx^*) = t(v) - \sum_{i=1}^n t (e_ie_i^*) = t(v) - \sum_{i=1}^n t(e_i^*e_i) = t(v) - \sum_{i=1}^n t(\ra(e_i)).$$

Now suppose that $t$ is faithful. Then, $0 < t(vv^*) = t (v)$ for all $v \in E^0$, which proves (4). Next, let $x$ and $u$ be as in (5). If $xx^* = u$, then $u=u^*$ and hence $$0 \leq (u-x^*x)(u-x^*x)^* = u-x^*x-x^*x+x^*ux = u-x^*x .$$ Since $t$ is faithful, $$t(u-x^*x) = t(u) - t(x^*x) = t(u)-t(xx^*) = 0$$ implies that $x^*x = u$. The converse follows by symmetry.

To show (6), first suppose that $E$ has at least two vertices, and that $p \in \path(E)$ is a cycle with an exit $e \in E^1$. We may assume that $\so(p)=\ra(p)=\so(e)$, and we shall denote this vertex by $v$. Let $w \in E^0 \setminus \{v\}$ be any vertex, and let $x = p+w$. Then $$x^*x = (p^*+w)(p+w) = p^*p + w = v+w.$$ Now, $v+w$ is an idempotent which clearly satisfies $x(v+w)=x=(v+w)x$ and $x^*(v+w)=x^*=(v+w)x^*$. Therefore, by (5) we have $$v+w = xx^* = (p+w)(p^*+w) = pp^*+w,$$ and hence, $v = pp^*$. But then $0 = e^*pp^* = e^*v=e^*$ contradicts our choice of $e$, implying that there cannot be a cycle with an exit.

Next, suppose that $E^0 = \{v\}$, and there is an edge $e \in E^1$ (necessarily a loop). Then $e^*e=v=1$ implies that $ee^*=v$, by (5). It follows that there cannot be another edge in $E$, for if $f \in E^1 \setminus \{e\}$, then $0=f^*ee^* = f^*v=f^*$, a contradiction. Thus, in all cases, no cycle in $E$ can have an exit.
\end{proof}

Let us pause to give an example showing that satisfying conditions (1)--(3) above is generally not sufficient for a trace $t$ to be positive, and satisfying conditions (1)--(6) is generally not sufficient for $t$ to be faithful.

\begin{example}
Let $K=\C$ be the field of the complex numbers, with the usual involution $*$, defined by $(a+bi)^*=a-bi$. Then the positive elements here are precisely the nonnegative real numbers. Also let $E$ be the one-loop graph, with $E^0 = \{v\}$ and $E^1 = \{e\}$, pictured below.
$$
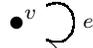
\xymatrix{{\bullet}^v \ar@(ur,dr)^e}$$ 
Using the fact that $e^*e=v=ee^*$, it is easy to see (and well-known) that $$\{v\} \cup \{e^n \mid n \geq 1\} \cup \{(e^*)^n \mid n \geq 1\}$$ is a basis for $L_\C(E)$ as a $\C$-vector space. In particular, $L_\C(E)$ is commutative. Now define $t : L_\C(E) \to \C$ by $t(v) = 1$, $t(e^n) = i^n$, $t((e^*)^n) = i^n$, and extend $\C$-linearly to all of $L_\C(E)$. Since $L_\C(E)$ is commutative, $t$ is a $\C$-linear trace. Since $t(v)=1$ and $v$ is the only vertex in $E$, conditions (1)--(6) of Proposition~\ref{prop-of-graph-trace-lemma} are clearly satisfied. However, letting $x = v+e$, we see that $$t(xx^*) = t((v+e)(v+e^*)) = t(v + e + e^* + v) = 2 + 2i.$$ Since $2 + 2i$ is not a nonnegative real number, $t$ is not positive, and consequently not faithful either.
\end{example}

We require one more ingredient to prove our main result about faithful traces on Leavitt path algebras; it is part of~\cite[Theorem 3.7]{AAPM}. In order to state it we first recall some terminology from~\cite{AAPM}. Given a graph $E$ and an infinite path $p=e_1e_2\hdots$ ($e_i \in E^1$), we say that $p$ is an \emph{infinite sink} if it has neither cycles nor exits. We also say that $p$ \emph{ends in a sink} if there is a subpath $e_ne_{n+1}\hdots$ ($n\geq 1$) which is an infinite sink, and that $p$ \emph{ends in a cycle} if there is subpath $e_ne_{n+1}\hdots$ ($n\geq 1$) and a cycle $c \in \path(E)\setminus E^0$ such that $e_ne_{n+1}\hdots = cc\hdots$.

\begin{theorem}[Abrams/Aranda Pino/Perera/Siles Molina] \label{APPM-theorem}
Let $K$ be a field and $E$ a countable row-finite graph. Then the following are equivalent.
\begin{enumerate}
\item[$(1)$] $E$ is a no-exit graph where every infinite path ends either in a sink or in a cycle.
\item[$(2)$] $L_K(E) \cong \bigoplus_{i\in I} \M_{\kappa_i} (K) \oplus \bigoplus_{i \in J} \M_{\kappa_i} (K[x,x^{-1}]),$ where $I$ and $J$ are countable sets, and each $\kappa_i$ is a countable cardinal.
\end{enumerate}
\end{theorem}

Assuming that $K$ has an involution, the isomorphism $$f : L_K(E) \cong \bigoplus_{i\in I} \M_{\kappa_i} (K) \oplus \bigoplus_{i \in J} \M_{\kappa_i} (K[x,x^{-1}])$$ constructed in the proof of $(1) \Rightarrow (2)$ of the above result is actually a $K$-linear $*$-isomorphism, viewing each $\M_{\kappa_i} (K)$ and $\M_{\kappa_i} (K[x,x^{-1}])$ as a $*$-ring via Lemmas~\ref{inv-matrix} and~\ref{inv-laurent}. 

More specifically, let $\{c_i\}_{i \in J}$ be all the cycles in $\path (E) \setminus E^0$, let $\{s_i\}_{i \in I_1}$ be all the sinks in $E^0$, and let $\{u_i\}_{i \in I_2}$ be all the infinite sinks in $E$ (where distinct $u_i$ have no edges in common, and where the starting vertex of each $u_i$ is fixed, though it can be chosen arbitrarily). Also, for each $i \in I_1$ let $\{p_{ij}\}_{j \in L_i}$ be all the paths that end in $s_i$, for each $i \in I_2$ let $\{q_{ij}\}_{j \in M_i}$ be all the paths that end in $u_i$ (such that $\ra(q_{ij})$ is the only vertex lying on both $q_{ij}$ and $u_i$), and for each $i \in J$ let $\{r_{ij}\}_{j \in N_i}$ be all the paths that end in $\so(c_i)$ but do not contain $c_i$ (where $\so(c_i)$ is a fixed, though arbitrary, vertex on $c_i$). It can be shown that $$\bigcup_{i \in I_1} \{p_{ij}p_{il}^* \mid j,l \in L_i\} \cup \bigcup_{i \in I_2} \{q_{ij}x_{ijl}q_{il}^* \mid j,l \in M_i\} \cup \bigcup_{i \in J} \{r_{ij}c_i^kr_{il}^* \mid j,l \in N_i, k\in \N\}$$ is a basis for $L_K(E)$, where $c_i^k$ is understood to be $\so(c_i)$ when $k=0$ and $(c_i^*)^{-k}$ when $k < 0$, and where, writing each $u_i$ as $u_i = e_{i1}e_{i2}\hdots$ ($e_{ij} \in E^1$), we define
$$x_{ijl} = \left\{ \begin{array}{cl}
e_{ik}e_{i,k+1}\hdots e_{im} & \text{if } \ra(q_{ij}) = \so(e_{ik}), \ra(q_{il}) = \ra(e_{im}), \text{ and } k\leq m\\
e_{ik}^*e_{i,k-1}^*\hdots e_{im}^* & \text{if } \ra(q_{ij}) = \ra(e_{ik}), \ra(q_{il}) = \so(e_{im}), \text{ and } k \geq m\\
\ra(q_{ij}) & \text{if } \ra(q_{ij})=\ra(q_{il})
\end{array}\right..$$
Then $f$ maps $L_K(E)$ to the ring $$\bigoplus_{i\in I_1} \M_{|L_i|} (K) \oplus \bigoplus_{i \in I_2} \M_{|M_i|} (K)  \oplus \bigoplus_{i \in J} \M_{|N_i|} (K[x,x^{-1}]),$$ via setting $f(p_{ij}p_{il}^*) = e_{jl} \in \M_{|L_i|} (K)$, $f(q_{ij}x_{ijl}q_{il}^*) = e_{jl} \in \M_{|M_i|} (K)$, and $f(r_{ij}c_i^kr_{il}^*) = x^ke_{jl} \in \M_{|N_i|} (K[x,x^{-1}])$, and then extending $K$-linearly to all of $L_K(E)$. Note that for all such elements we have $$f((p_{ij}p_{il}^*)^*) = f(p_{il}p_{ij}^*) = e_{lj} = e_{jl}^* = f(p_{ij}p_{il}^*)^*,$$ $$f((q_{ij}x_{ijl}q_{il}^*)^*) = f(q_{il}x_{ijl}^*q_{ij}^*) = e_{lj} = e_{jl}^* = f(q_{ij}x_{ijl}q_{il}^*)^*, \text{ and}$$ $$f((r_{ij}c_i^kr_{il}^*)^*) = f(r_{il}c_i^{-k}r_{ij}^*) = x^{-k}e_{lj} = (x^ke_{jl})^* = f(r_{ij}c_i^kr_{il}^*)^*,$$ from which it follows that $f(y^*) = f(y)^*$ for all $y \in L_K(E)$. 

We also observe that the proof of Theorem~\ref{APPM-theorem} does not rely on $E$ being countable to work. Thus, with trivial modifications, the proof actually shows the following.

\begin{corollary} \label{LPA-matrix}
Let $K$ be a field with an involution, and let $E$ be a row-finite graph. Then the following are equivalent.
\begin{enumerate}
\item[$(1)$] $E$ is a no-exit graph where every infinite path ends either in a sink or in a cycle.
\item[$(2)$] $L_K(E) \cong \bigoplus_{i\in I} \M_{\kappa_i} (K) \oplus \bigoplus_{i \in J} \M_{\kappa_i} (K[x,x^{-1}]),$ for some sets $I, J$ and cardinals $\kappa_i$.
\end{enumerate}
Moreover, the isomorphism in $\, (2)$ can be taken to be a $K$-linear $\, *$-isomorphism.
\end{corollary}

We are now ready to prove our main result about faithful traces on Leavitt path algebras.

\begin{theorem} \label{faithful-LPA}
Let $K$ be a field with a positive definite involution, and let $E$ be a row-finite graph where every infinite path ends either in a sink or in a cycle. Then the following are equivalent.
\begin{enumerate}
\item[$(1)$] There is a faithful $K$-linear trace $t: L_K(E) \to K$.
\item[$(2)$] There is a faithful $K$-linear trace $t: L_K(E) \to R$ for some $K$-algebra $R$ with an involution.
\item[$(3)$] There is a faithful $K$-linear minimal trace $t: L_K(E) \to R$ for some $K$-algebra $R$ with an involution.
\item[$(4)$] There is a faithful minimal trace $t: L_K(E) \to R$ for some $*$-ring $R$.
\item[$(5)$] There is a faithful trace $t: L_K(E) \to R$ for some $*$-ring $R$.
\item[$(6)$] $E$ is a no-exit graph.
\end{enumerate}
\end{theorem}

\begin{proof}
By Proposition~\ref{prop-of-graph-trace-lemma}(6), (5) implies (6). The other implications in the following diagram are tautological. 
$$\xymatrix@=.8pc{{(1)} \ar@{=>}[r] & {(2)} \ar@{=>}[rd]  & & \\
& & {(5)} \ar@{=>}[r]  & {(6)} \\ 
{(3)} \ar@{=>}[r] & {(4)} \ar@{=>}[ru]  & &}$$
We shall conclude the proof by showing that (6) implies (1) and (3).

Suppose that (6) holds. Since every infinite path in $E$ ends either in a sink or in a cycle, and $E$ is a row-finite no-exit graph, by Corollary~\ref{LPA-matrix}, $L_K(E)$ is isomorphic to $\bigoplus_{i\in I} R_i \oplus \bigoplus_{i \in J} S_i$ via a $K$-linear $*$-isomorphism, where each $R_i$ is of the form $\M_\kappa (K)$ (for some cardinal $\kappa$) and each $S_i$ is of the form $\M_\kappa (K[x,x^{-1}])$ (for some cardinal $\kappa$). 

By Lemma~\ref{inv-matrix}, each $R_i$ admits a faithful $K$-linear trace $t_i : R_i \to K$, which is in addition minimal, by Corollary~\ref{usual-tr-min}. Also, extending the involution from $K$ to $K[x,x^{-1}]$, as in Lemma~\ref{inv-laurent}(1), and then to $\M_\kappa (K[x,x^{-1}])$, as in Lemma~\ref{inv-matrix} (for any nonzero cardinal $\kappa$), makes the usual trace $\tr: \M_\kappa (K[x,x^{-1}]) \to K[x,x^{-1}]$ faithful, since, by Lemma~\ref{inv-laurent}(2), the above involution on $K[x,x^{-1}]$ is positive definite. Thus, each $S_i$ admits a faithful $K$-linear trace $u_i : S_i \to K[x,x^{-1}]$, which is again minimal, by Corollary~\ref{usual-tr-min}. Let $$f : \bigoplus_{i \in I} R_i \oplus \bigoplus_{i \in J} S_i \to \bigoplus_{i \in I} K \oplus \bigoplus_{i \in J} K[x,x^{-1}]$$ act on each direct summand as the corresponding map $t_i$ or $u_i$. Then it is easy to see that $f$ gives a faithful $K$-linear minimal trace on $\bigoplus_{i \in I} R_i \oplus \bigoplus_{i \in J} S_i$, and hence on $L_K(E)$. Thus (6) implies (3).

Now, letting $t : K[x,x^{-1}] \to K$ be as in Lemma~\ref{inv-laurent}(3), we see that for each $i \in J$, $t \circ u_i : S_i \to K$ is a faithful $K$-linear trace, by Lemma~\ref{prod-lemma}(1). Therefore, setting $t_i = t \circ u_i$ for each $i \in J$, and letting $$h : \bigoplus_{i \in I} R_i \oplus \bigoplus_{i \in J} S_i \to \bigoplus_{i \in I \cup J} K$$ act on each direct summand as the corresponding map $t_i$, again gives a faithful $K$-linear trace on $\bigoplus_{i \in I} R_i \oplus \bigoplus_{i \in J} S_i$, and hence on $L_K(E)$. Further, by Lemma~\ref{prod-lemma}(2), there is a faithful $K$-linear trace $\bigoplus_{i \in I \cup J} K \to K$. Composing this trace with $h$ (and the aforementioned $*$-isomorphism) then gives a faithful $K$-linear trace $L_K(E) \to K$, by Lemma~\ref{prod-lemma}(1). Therefore (6) implies (1).
\end{proof}

We conclude with three examples showing the necessity of the assumptions on $K$ and $E$ in the previous theorem.

\begin{example}
Let $E$ be the graph consisting of two vertices $v,w$ and one edge $e$, having source $v$ and range $w$, as pictured below.  
$$\xymatrix{\bullet^{v} \ar[r]^{e} 
& \bullet^{w} }$$
Then $E$ is a no-exit row-finite graph with no infinite paths. Also, let $K=\C$ be the field of the complex numbers, endowed with the identity involution $*$. Since $1\cdot 1^*+i\cdot i^*=0$, the involution is not positive definite. We shall show that there are no faithful $\C$-linear traces $t : L_\C (E) \to \C$, and hence that here condition (6) of Theorem~\ref{faithful-LPA} is satisfied, but not condition (1).

Let $t : L_\C (E) \to \C$ be any trace. Then $$t((v+iw)(v+iw)^*) = t(v-w) = t(ee^*)-t(e^*e) = t(ee^*) - t(ee^*) = 0,$$ and hence $t$ cannot be faithful.
\end{example}

\begin{example}
Let $K=\R$ be the field of the real numbers, endowed with the identity involution, which is easily seen to be positive definite. (Here a nonzero element is positive, in the sense of Definition~\ref{invol-def}, precisely when it is positive in the usual sense.) Also let $E$ be the graph consisting of two vertices $v,w$ and infinitely many edges $e_1, e_2, e_3, \dots$, each having source $v$ and range $w$, as pictured below.  
$$
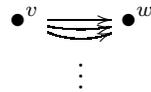
\xymatrix{\bullet^{v} \ar[r] \ar@/_.3pc/[r] \ar@/_.6pc/[r]_{\vdots} 
& \bullet^{w} }$$
Then $E$ is a no-exit graph with no infinite paths, but it is not row-finite. We shall show that there are no faithful $\R$-linear traces $t : L_\R (E) \to \R$, and hence that here condition (6) of Theorem~\ref{faithful-LPA} is satisfied, but not condition (1). 

Suppose that, on the contrary, $t : L_\R (E) \to \R$ is a faithful $\R$-linear trace. Then, by Proposition~\ref{prop-of-graph-trace-lemma}(4), we have $t(v), t(w) > 0$, and by Proposition~\ref{prop-of-graph-trace-lemma}(3), we see that for all $n \geq 1$, $$t(v) - nt(w) = t(v) - \sum_{i=1}^nt(\ra(e_i)) \geq 0.$$ Thus, $t(v) \geq nt(w) >0$ for all $n \geq 1$, which is absurd.
\end{example}

\begin{example}
Again, let $K=\R$ be the field of the real numbers, endowed with the identity involution. Now let $E$ be the following graph.
$$
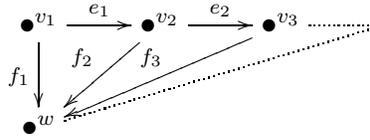
\xymatrix{{\bullet}^{v_1} \ar[r]^{e_1} \ar[d]_{f_1} & {\bullet}^{v_2} \ar[r]^{e_2} \ar[dl]_{f_2} & {\bullet}^{v_3} \ar[dll]_{f_3} \ar@{.}[r] & \ar@{.}[dlll]\\ 
{\bullet}^{w}}$$
Then $E$ is a row-finite no-exit graph, but with an infinite path that does not end in a sink or a cycle, namely $e_1e_2e_3\dots$. Again, we shall show that there are no faithful $\R$-linear traces $t : L_\R (E) \to \R$, and hence that condition (6) of Theorem~\ref{faithful-LPA} is satisfied, but not condition (1). 

Suppose that, on the contrary, $t : L_\R (E) \to \R$ is a faithful $\R$-linear trace. For all $i \geq 1$ we have $v_i =  e_i e_i^* + f_i f_i^*$ and hence also $$t(v_i) = t(e_i^*e_i) + t(f_i^*f_i) = t(v_{i+1}) + t(w).$$  Iterating this relation gives $t(v_1) = t(v_{i+1}) + it(w)$ for all $i \geq 1$. This produces the desired contradiction upon noting that $t(v_{i+1}), t(w) > 0$, by Proposition~\ref{prop-of-graph-trace-lemma}(4).
\end{example}

\end{document}